\documentclass[11pt]{amsart}
\usepackage{color}
\usepackage{amsmath,amsthm,amsfonts,amssymb}
\usepackage{amscd,graphics,graphicx}
\usepackage{url}
\usepackage{hyperref}

\numberwithin{equation}{section}

\theoremstyle{plain}

\newtheorem{theorem}[subsection]{Theorem}
\newtheorem{proposition}[subsection]{Proposition}
\newtheorem{lemma}[subsection]{Lemma}

\theoremstyle{definition}

\newtheorem{definition}[subsection]{Definition}

\newcommand\ds{\displaystyle}

\newcommand{\Z}{\mathbb Z}
\newcommand{\C}{\mathbb C}

\newcommand{\N}{\mathbb N}
\newcommand{\R}{\mathbb{R}}

\newcommand{\odd}{{\rm odd}}
\newcommand{\ep}{\varepsilon}
\newcommand{\CC}{\mathcal{C}}

\newcommand\psit{{\widetilde {\psi}}}

\title[Local $L^2$-regularity of Riemann's  Fourier  series]{Local $L^2$-regularity of \\Riemann's  Fourier  series}
\author[St\'ephane Seuret]{St\'ephane Seuret$^\ast$}

\address {St\'ephane SEURET, Universit\'e Paris-Est, LAMA (UMR 8050), UPEMLV, UPEC, CNRS, F-94010, Cr\'eteil, France}
\email{seuret@u-pec.fr}

\author{Adri\'an UBIS}
\address{Adri\'an UBIS, Departamento de Matem\'aticas, Facultad de Ciencias, Universidad Aut\'onoma de Madrid, 28049 Madrid, Spain}
\email{adrian.ubis@uam.es}
\thanks{$^\ast$ Research partially supported by the ANR project MUTADIS, ANR-11-JS01-0009.  }

\begin{document}

 \begin{abstract}
 We are interested in the convergence and the local regularity  of the lacunary Fourier series $F_s(x) = \sum_{n=1}^{+\infty} \frac{e^{2i\pi n^2 x}}{n^s}$. In the 1850's, Riemann introduced the series $F_2$ as a possible example of nowhere differentiable function, and the study of this function has  drawn the interest of many mathematicians since then. We focus on the case when $1/2<s\leq 1$, and we prove that $F_s(x)$ converges when $x$ satisfies  a  Diophantine condition. We also study the $L^2$- local regularity of $F_s$, proving that the local $L^2$-norms have different behaviors around different $x$, according again to Diophantine conditions on $x$.
 \end{abstract}
  
\maketitle

\section{Introduction}

Riemann introduced in 1857 the Fourier series
$$R(x) = \sum_{n =1}^{+\infty} \frac{\sin(2\pi n^2  x)}{n^2}$$
as a possible example of continuous but nowhere differentiable function. Though it is not the case ($R$ is differentiable at rationals $p/q$ where $p$ and $q$ are both odd \cite{gerver}), the study of this function has, mainly because of its connections with  several domains: complex 
analysis, harmonic analysis, Diophantine approximation, and dynamical systems \cite{Hardy,HL,gerver,Dui,Itatsu,jaffard} and more recently \cite{chamizo_ubis_gaps,chamizo_ubis_polynomial,rivoal_seuret}.

In this article, we study the local regularity of the series
\begin{equation}
\label{defFs}
F_s(x) = \sum_{n=1}^{+\infty} \frac{e^{2i\pi n^2 x}}{n^s}
\end{equation}
when $s\in (1/2,1)$. In this case, several questions arise before considering its local behavior. First it does not converge everywhere, hence one needs to characterize its set of convergence points{\color{red};} this question was studied in \cite{rivoal_seuret}, and we will first  find a slightly more precise characterization. Then, if one wants to characterize the local regularity of a (real) function, one classically studies the  pointwise H\"older exponent defined for a locally bounded function $f:\R\to\R$ at a point $x$ by using the functional spaces  $C^\alpha (x)$:   $f\in C^\alpha(x)$ when there exist a constant $C$ and a polynomial $P$ with degree less than $\lfloor\alpha\rfloor$ such that, locally around $x$ (i.e. for small $H$), one has
$$  \big\|(f(\cdot)-P(\cdot-x) ) {\bf 1\!\!1}_{B(x,H)}\|_{\infty}:= \sup( |f(y)-P(y-x)| : y\in B(x,H)) \leq C H^\alpha,$$
where $B(x,H) = \{y\in \R: |y-x|\leq H\}$. Unfortunately these spaces are not appropriate for our context since $F_s$ is nowhere locally bounded (for instance, it diverges at every irreducible rational $p/q$ such that $q\neq 2\times $odd).

Following Calderon and Zygmund in their study of local behaviors of solutions of elliptic PDE's  \cite{CZ}, it is natural to introduce in this case the pointwise $L^2$-exponent defined as follows.

\begin{definition}
Let $f:\R \to \R$ be a function belonging to $L^2 (\R)$, $\alpha \geq 0$ and $x\in \R$. The function $f$ is said to belong to $C^\alpha_2(x)$ if there exist a constant $C$ and a polynomial $P$  with degree less than $\lfloor\alpha\rfloor$ such that, locally around $x$ (i.e. for small $H>0$), one has
$$ \left(  \frac{1}{H} \int_{B(x,H)} \big|f(h)-P(h-x \big)| ^2dh  \right) ^{1/2}\leq C H^\alpha.$$
Then, the pointwise $L^2$-exponent of $f$ at $x$ is
$$\alpha_f(x) = \sup\big\{\alpha \in \R: f\in C^\alpha_2(x) \big\}.$$
\end{definition}

This definition makes sense for the series $F_s$ when $s\in (1/2,1)$, and are based on a natural generalizations of the spaces $C^\alpha(x)$ be replacing the $L^\infty$ norm by the $L^2$ norm.    The pointwise $L^2$-exponent has been studied for instance in \cite{JM}, and is always greater than -1/2 as soon as $f\in L^2 $.

\medskip

Our goal is to perform the multifractal analysis of the series $F_s$. In other words, we aim at computing the Hausdorff dimension, denoted by $\dim$ in the following,  of the level sets of the pointwise $L^2$-exponents.

\begin{definition}
Let $f:\R \to \R$ be a function belonging to $L^2 (\R)$. The $L^2$-multifractal spectrum $d_f : \R^+ \cup\{+\infty\} \to \R^+\cup \{-\infty\}$ of $f$ is the mapping 
$$ d_f( \alpha) := \dim E_f(\alpha),$$
where the iso-H\"older set $E_f(\alpha) $ is
$$E_f(\alpha) := \{x\in \R: \alpha_f(x) =\alpha\} .$$ 
\end{definition}
By convention one sets $d_f(\alpha) = -\infty$ if $E_f( \alpha) =\emptyset$.

\medskip

Performing the multifractal analysis consists in computing its $L^2$-{multifractal} spectrum. This provides us with a very precise description of the distribution of the local $L^2$-singularities of $f$.
In order to state our result, we need to introduce some notations.

\begin{definition}
Let $x$ be an irrational number, with convergents $(p_j/q_j)_{j\ge 1}$. Let us define
\begin{equation}  
\label{defconvergents}
 x-\frac{p_j}{q_j}= h_j, \qquad |h_j|=q_j^{-r_j}
\end{equation}
with $2\le r_j <\infty$. Then  the approximation rate of $x$  is defined by
\[
 r_{\odd}(x)=\overline{\lim} \{r_j: q_j\neq 2*\odd \}.
\]
\end{definition}
This definition always makes sense because if $q_j$ is even, then $q_{j+1}$ and $q_{j-1}$ must be odd (so cannot be equal to $2*\odd$). Thus, we always have $2\le r_{\odd}(x) \le +\infty$. It is classical  that one can compute the Hausdorff dimension of the set of points with  the Hausdorff dimension of the points $x$ with a given approximation rate $r\geq 2$:
\begin{equation}\label{jarnik}
  \mbox{for all } r\geq 2, \ \ \ \dim \{x\in \R: r_{\odd}(x)=r\} = \frac{2}{r}.
\end{equation}

When  $s>1$, the series  $F_s$ converges, and the multifractal spectrum of $ F_s$ was computed by S. Jaffard in \cite{jaffard}. For instance, for the classical Riemann's series $F_2$, one has
$$d_{F_2}(\alpha) = \begin{cases}  \ 4\alpha-2 & \mbox { if } \alpha\in [1/2,3/4], \\  \ \ \ 0 & \mbox { if } \alpha=3/2, \\  \ -\infty & \mbox { otherwise.}
\end{cases}
$$

Here our aim is to somehow extend this result to the range $1/2<s\le 1$.  The convergence of the series $F_s$ is described by our first theorem.

\begin{theorem}\label{convergence}
Let $s\in (1/2,1]$, and let $x\in (0,1)$ with convergents $(p_j/q_j)_{j\ge 1}$. We set for every $j\geq 1$
$$\delta_j= \begin{cases} \ \ \ \  \ \  1 & \mbox{ if }s\in(1/2,1)\\  \ \log (q_{j+1}/q_j) & \mbox{ if } s=1,\end{cases}$$
and 
\begin{equation}
\label{S<+infty}
\Sigma_s(x)= \sum_{j: \, q_j\neq 2*\odd} \delta_j\sqrt{\frac{q_{j+1}}{(q_jq_{j+1})^{s}}}.
\end{equation}

\begin{enumerate}
\item
$F_s(x)$ converges whenever $\frac{s-1+1/r_{\odd}(x)}{2} >0$. In fact, it converges whenever $\Sigma_s(x)<+\infty$.
 \item
$F_s(x)$ does not converge if $\frac{s-1+1/r_{\odd}(x)}{2}<0$. In fact, it does not converge whenever
\[
 \overline{\lim}_{j:\, q_j\neq 2*\odd} \,\,  \delta_j\sqrt{\frac{q_{j+1}}{(q_jq_{j+1})^{s}}} >0,
\]
\end{enumerate}
\end{theorem}

In the same way we could extend this results to rational points $x=p/q$, by proving that $F_s(x)$ converges for $q\neq 2*\odd$ and does not for $q\neq 2*\odd$.  Observe that the convergence of $\Sigma_s(x)$ implies that $r_{odd}(x) \leq \frac{1}{1-s}$. Our result asserts that $F_s(x)$ converges as soon as $r_{\odd}(x) < \frac{1}{1-s}$, and also when $r_{\odd}(x) = \frac{1}{1-s}$ when $\Sigma_s(x)<+\infty$.

\begin{center}
\begin{figure}
  \includegraphics[width=6.5cm,height = 5.0cm]{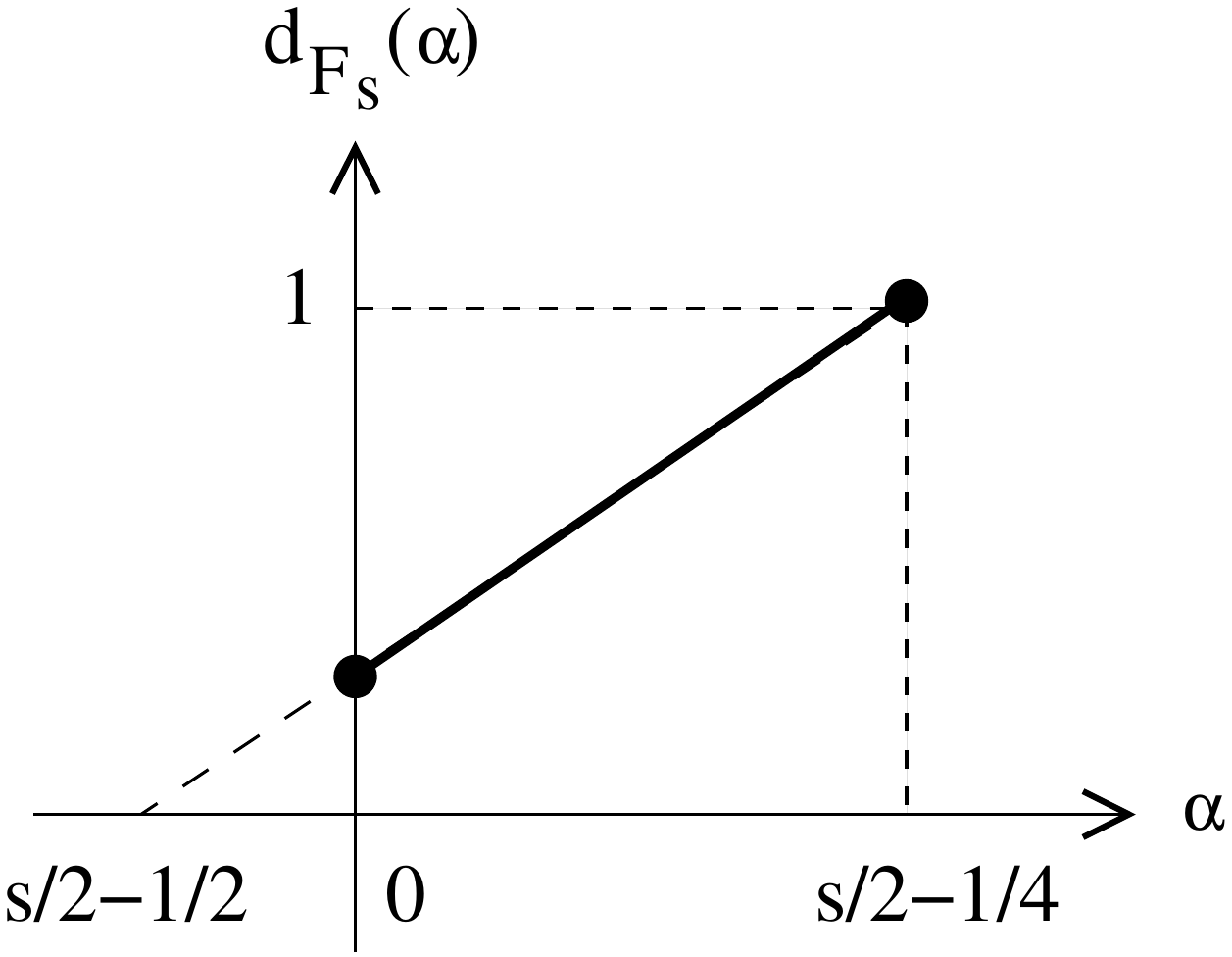} 
\caption{$L^2$-multifractal spectrum of $F_s$ } \label{fig1}
 \end{figure}
\end{center}
 
Jaffard's result is then extended in the following sense:
\begin{theorem}
\label{smoothness}
Let $s\in (1/2,1]$.
\begin{enumerate}
\item
For every  $x$  such that $\Sigma_s (x)< +\infty$, one has $\ds \alpha_{F_s}(x)  = \frac{s-1+1/r_{\odd}(x)}{2}$.
\item
For every $\alpha\in [0,  s/2-1/4]$
$$d_{F_s}(\alpha) = 4\alpha +2 - 2s.$$ 
\end{enumerate}
\end{theorem}

The second part of Theorem \ref{smoothness} follows directly from the first one. {Indeed, using part (i) of Theorem \ref{smoothness} and (\ref{jarnik}),} one gets
\begin{eqnarray*}
d_{F_s}(\alpha)  & =  &\dim \left \{x: \frac{s-1+1/r_{\odd}(x)}{2}=\alpha \right\}= \dim  \{x: r_{\odd}(x)= (2\alpha+1-s)^{-1}\}\\
& = & \frac{2}{  (2\alpha+1-s)^{-1}} = 4\alpha+2-2s.
\end{eqnarray*}

\medskip

The paper is organized as follows. Section 2 contains some notations and preliminary results.  In Section 3, we obtain other formulations for $F_s$ based on Gauss sums, and we get first estimates on the increments of the partial sums of the series $F_s$. Using these results, we prove Theorem \ref{convergence} in Section 4. Finally, in Section 5, we use the previous estimates to obtain upper and lower bounds for the local $L^2$-means of the series $F_s$, and compute in Section 6 the local $L^2$-regularity exponent of $F_s$ at real numbers $x$ whose Diophantine properties are controlled, {namely we prove Theorem \ref{smoothness}.}

\medskip

{Finally, let us mention that theoretically the $L^2$-exponents  of a function $f\in L^2(\R)$ take values in the range $[-1/2,+\infty]$, so they may have negative values. We believe that this is the case at points $x$ such that $r_\odd(x) >\frac{1}{1-s}$, so that in the end the entire $L^2$-multifractal spectrum of $F_s$ would be $d_{F_s}(\alpha) = 4\alpha +2 - 2s$ for all $\alpha \in [  s/2-1/2,s/2-1/4 ]$.

Another remark is that for a given $s\in (1/2,1)$, there is an optimal $p>2$ such that $F_s$ belongs locally to $L^p$, so that the $p$-exponents (instead of the 2-exponents)  may carry some interesting information about the local behavior of $F_s$.

}

\section{Notations and first properties}

In all the proofs, $C$ will denote a constant that does not depend on the variables involved in the equations.

For two real numbers $A,B\geq 0$, the notation $ A \ll B$ means that $A\leq C B$ for some constant $C>0$ independent of the variables in the problem.

\medskip

In Section 2 of \cite{chamizo_ubis_polynomial} (also in \cite{chamizo_ubis_gaps}), the key point to study the local behavior of the Fourier series $F_s$ was to obtain an explicit formula for  $F_s(p/q+h)-F_s(p/q)$    in the range $1<s<2$; this formula was just a twisted version of the one known for the Jacobi theta function. In our range $1/2<s\le 1$, such a formula cannot exist because of the convergence problems, but we will get some truncated versions of it in order to prove Theorems \ref{convergence} and \ref{smoothness}.

Let us introduce the partial sum
$$F_{s,N}(x)=  \sum_{n=1}^{ N} \frac{e^{2i\pi n^2 x}}{n^s}.$$

For any $H\neq 0$ let $ \widetilde \mu_H$ be the probability  measure defined by 
\begin{equation}
\label{defmutilde}
 \widetilde \mu_H (g)=\int_{\CC(H)} g(h) \, \frac{dh}{{2}H} ,
\end{equation}
where $\CC(H)$ is the annulus $\CC(H) = [-2H,-H]\cup[H,2H]$.
\begin{lemma}
Let $f:\R\to\R$ be a function in $L^2(\R)$, and $x\in \R$. If $\alpha_f(x) <1$, then
\begin{equation} 
\label{defalphaf}
\alpha_f(x) = \sup \left \{ \beta\in[0,1): \exists \, C>0, \exists \, f_x\in\R\ \|f(x+\cdot)-f_x\|_{L^2( \widetilde \mu_H)} \leq  C |H|^{\beta}  \right \}.
\end{equation}
\end{lemma}

\begin{proof}
Assume that $\alpha:=\alpha_f(x) <1$. Then, for every $\ep>0$, there exist $C>0$ and a real number  $f_x \in \R$ such that for every $H>0$  small enough
$$ \left(  \frac{1}{H} \int_{B(x,H)} \big|f(h)-  f_x\big| ^2dh  \right) ^{1/2}\leq C H^{\alpha -\ep}.$$
Since $\CC(H) \subset B(x,2H)$, one has
\begin{equation}
\label{eq1}
\big\|f(x+\cdot)-f_x  \big\|_{L^2( \widetilde \mu_{H})} =   \! \left( \int_{\CC(H)} \big|f(x+h)-  f_x\big| ^2  \, \frac{dh}{{ 2}H}  \right)^{1/2}  \!\!\! \!\!\! \leq  C |2H|^{\alpha -\ep}   \! {  \ll} | H|^{\alpha -\ep}  .
\end{equation}
Conversely, if \eqref{eq1} holds for every $H>0$, then the result follows from the fact that  $B(x,H) = x+ \bigcup_{k\geq 1} \CC( H/2^k)$.
\end{proof}

So, we will use equation \eqref{defalphaf} as definition of the local $L^2$ regularity.

\medskip

It is important to notice that the frequencies in different ranges are going to behave differently. Hence, it is better to look at $N$ within dyadic intervals. Moreover, it will be easier to deal with smooth pieces. This motivates the following definition.

\begin{definition}
Let $N\geq 1$ and let $\psi:\R\to\R$ be a $C^\infty$ function with support included in $[1/2,2]$. One introduces the series
\begin{equation}
\label{defFpsi}
F_{s,N} ^{\psi}(x)=\sum_{n=1}^{+\infty}  \frac{e^{2i\pi n^2 x}}{n^s} \psi  \left(\frac nN \right ),
\end{equation}
and for $R>0$ one also sets
$$w_R^{\psi}(t)=e^{2i\pi Rt^2}\psi(t)$$
and 
\[
 E_N^{\psi}(x)=\frac 1N \sum_{n=1}^{+\infty}  w_{N^2 x}^{\psi} \left(\frac nN \right ),
\]
\end{definition}

For the function $\psi_s(t)=t^{-s} \psi(t)$ (which is  still $C^{\infty}$ with  support in $(1/2,2)$) it is immediate to check that 
\begin{equation}
\label{eq2}
 F_{s,N}^{\psi}(x)=N^{1-s} E_N^{\psi_s}(x).
 \end{equation}

\section{Summation Formula for $F_{s,N}$ and $F_{{\color{red}s},N}^\psi$}

\subsection{ {Poisson Summation.}}

Let $p,q$ be coprime integers, with $q>0$.  In this section we obtain some  formulas for $F(p/q+h)-F(p/q)$ with $h>0$. This is not a restriction, since 
\begin{equation}\label{symmetry}
 F_s\left(\frac pq-h \right)=\overline{F_s\left(\frac{-p}q+h \right)}.
\end{equation}

We are going to write a summation formula for $E_N^{\psi}(p/q+h)$, with $h>0$. 

\begin{proposition}
We have
\begin{equation}\label{poisson}
E_N^{\psi}\left(\frac pq +h\right)=\frac{1}{\sqrt q} \sum_{m \in \Z} \theta_m  \cdot \widehat {w_{N^2 h}^{\psi}} \left (\frac{Nm}{q} \right).
\end{equation}
where $\ds\widehat{f}(\xi) = \int_\R f(t)e^{-2i\pi t\xi}dt$ stands for the Fourier transform of $f$ and $(\theta_m)_{m\in \Z}$ are some complex numbers whose modulus is bounded by $\sqrt{2}$.
\end{proposition}
\begin{proof} 
We begin by splitting the series into arithmetic progressions
\begin{eqnarray*}
 E_N^{\psi}\left(\frac pq +h\right) & = & \frac 1N \sum_{n=1}^{+\infty}   e^{2i\pi 
n^2\frac pq}  \cdot w_{N^2 h}^{\psi} \left(\frac nN \right )   \\
& = & \sum_{b =0}^{q-1} e^{2i\pi b^2 \frac pq}  \sum_{\substack{n =1:\\ n{\equiv}  b \! \!\! \mod q}}^{+\infty} \frac 1N  \, w_{N^2 h}^{\psi}\left(\frac nN \right )  .
\end{eqnarray*}
Now we apply Poisson Summation to the inner sum to get 
\begin{eqnarray*} 
\sum_{\substack{n =1:\\ {n \equiv} b \! \!\! \mod q}}^{+\infty}   \frac 1N  w_{N^2 h}^{\psi}\left(\frac nN \right )    & = &   \sum_{ n =1 }^{+\infty}  \frac 1N  w_{N^2 h}^{\psi}\left(\frac {b+n q}N \right )\\
& = & \sum_{m\in \Z} \frac 1q e^{2i\pi \frac{bm}{q}} \cdot  \widehat {w_{N^2 h}^{\psi} } \left (\frac{Nm}q \right)
\end{eqnarray*}
This yields 
\[
 E_N^{\psi}\left(\frac pq +h\right)=\frac 1q\sum_{m\in \Z}  \tau_m  \cdot \widehat {w_{N^2 h}^{\psi}}    \left(\frac{Nm}q \right)
\]
with
\[
 \tau_m=\sum_{b=0}^{q-1} e^{2i\pi \frac{pb^2 + mb}q}.
\]
This  term $\tau_m$ is a Gauss sum. One has the following bounds:
\begin{itemize}
\item
 for every $m\in\Z$, $\tau_m=\theta_m \sqrt{q}$ with $\theta_m:= \theta_m(p/q)$  satisfying
 $$0\le |\theta_m| \le \sqrt 2.$$
\item
if $q=2*\odd$, then $\theta_0=0$.
\item
if $q\neq 2*\odd$, then $ 1\leq \theta_0\leq \sqrt{2}$. 
\end{itemize}
Finally, we get the summation formula \eqref{poisson}.
\end{proof}

\subsection{Behavior of the Fourier transform of $w^\psi_R$  } 

To use  formula \eqref{poisson}, one needs to understand the behavior of the Fourier transform of $w^\psi_R$  
\[
 \widehat {w_R^{\psi}} (\xi)= \int_\R \psi(t) e^{2i\pi (Rt^2-\xi t)} \, dt.
\]

On one hand, we have the trivial bound $|\widehat {w_R^{\psi} }(\xi) | \ll 1$ since $\psi$ is $C^\infty$, bounded by 1 and  compactly supported. On the other hand,  one has 
{
\begin{lemma}\label{fourier}
Let $R>0$ and $\xi\in \R$. Let $\psi$ be a  $C^{\infty}$ function  compactly supported inside $[1/2,2]$. Let us introduce the mapping $ g_R^{\psi}:\R\to\C$
\[
 g_R^{\psi}(\xi)= e^{i\pi/4} \frac{e^{-i\pi \xi^2/(2R)}}{\sqrt{2R}} \psi \left(\frac{\xi}{2R}\right).
\]
Then 
one has
\begin{equation}\label{eq3}
 \widehat{ w_R^{\psi}}(\xi)=  g_R^{\psi}(\xi)+  \mathcal{O}_{\psi}\left(\frac{\rho_{R,\xi}}{\sqrt{R}}+ \frac{1}{(1+R+|\xi|)^{3/2}}\right),
\end{equation}
with $\rho_{R,\xi} = \begin{cases} \mbox{ 1  \ \ \ if $\xi/2R \in [1/2,2]$ and $R<1$,}\\
 \mbox{ 0 \ \ \  otherwise}.\end{cases}$\!\!\!\!  \!\!\!.
 Moreover, one has
\begin{equation}\label{fourier_substract}
 \widehat{ w_{2R}^{\psi}}(\xi)- \widehat{ w_{R}^{\psi}}(\xi) =  g_{2R}^{\psi}(\xi)-g_{R}^{\psi}(\xi)+ \mathcal{O}_{\psi} \left( \rho_{R,\xi}\sqrt{R}+\frac R{(1+R+|\xi|)^{5/2}} \right),
\end{equation}

Moreover, the constant implicit in $\mathcal{O}_{\psi}$ depends just on the $L^\infty$-norm of a finite number of derivatives of $\psi$.
\end{lemma}

\begin{proof}
For $\xi/2R\not\in [1/2,2]$ the upper bound (\ref{eq3}) comes just from integrating by parts several times; for $\xi, R\ll 1$ the bound (\ref{eq3}) is trivial.  The same properties hold true for the upper bound in 
  (\ref{fourier_substract}).
  
  \smallskip
  
 Let us assume $\xi/2R\in [1/2,2]$, and $R>1$.  The  lemma is just a consequence of the stationary phase theorem. Precisely, Proposition 3 in Chapter VIII of \cite{stein} (and the remarks thereafter) implies that  for some suitable functions $f,g:\R\to\R$, and if $g$ is such that  $g'(t_0)=0$ at a unique point $t_0$,  if one sets 
$$
S(\lambda)=\int_\mathbb{R} f(t) e^ {i\lambda g(t)} \,dt -  \sqrt{\frac{2\pi}{-i\lambda g''(t_0)}} f(t_0) e^{i\lambda g(t_0)},
$$
then $|S(\lambda)|\ll \lambda^{-3/2} $ 
and also $|S'(\lambda)| \ll \lambda^ {-5/2}$, where  the implicit constants depend just on upper bounds for some derivatives of $f$ and $g$, and also on a lower bound for $g''$. 

In our case, we can apply it with  $f=\psi$, $\lambda=R$, and $g(t)=2\pi(t^2-\xi/\lambda)$ to get precisely (\ref{eq3}).

\medskip
With the same choices for $f$ and $g$, by applying the Mean Value Theorem and our bound for $S'(\lambda)$, we finally  obtain formula  (\ref{fourier_substract}).
 \end{proof}
 }


\subsection{Summation formula for the partial series  $F_{s,N}$  }

This important formula  will be useful to study the convergence of $F_s(x)$.

\begin{proposition}\label{summation}
Let $p,q$ be two coprime integers. For $N\geq q$ and $0\leq h \leq q^{-1}$, we have
\begin{eqnarray*}
 F_{s,2N} \left(\frac pq +h \right)-F_{s,N} \left(\frac pq +h \right) & =  & \frac{   \theta_0   }{\sqrt q} \displaystyle \int_N^{2N} \frac{e^{2i\pi ht^2}}{t^s} dt\\
 &&+ G_{s,2N}(h)-G_{s,N}(h)+\mathcal{O}( N^{\frac 12-s}\log q),
\end{eqnarray*}
where 
\begin{equation}
\label{DEFGN}
G_{s,N}(h)=  (2hq)^{s-\frac 12} e^{i\pi/4} \sum_{ m=1}^{\lfloor 2Nhq \rfloor} \frac{\theta_m}{m^s} e^{-i\pi \frac{m^2}{q^2 h}}.
\end{equation}
\end{proposition}

Pay attention to the fact that $G_{s,N}$ depends on $p$ and $q$. We omit this dependence in the notation  for clarity.
\begin{proof}
We  can write 
\[
 F_{s,2N} (x)-F_{s,N} (x)=F_{s,N}^{{\bf 1\!\!\!1}_{[1,2]}}(x).
\]
Hence, we would like to use the formulas proved in the preceding section, but those formulas apply only to compactly supported $C^\infty$ functions. We thus decompose the {indicator} function ${\bf 1\!\!\!1}_{[1,2]}$ into a countable sum of $C^\infty$ functions, as follows. Let us consider $\eta$, a $C^{\infty}$ function with support $[1/2,2]$ such that
\[
 \eta(t)=1-\eta(t/2)     \qquad 1\le t\le 2.
\]
Then, the function
\[
 \psi(t)=\sum_{k\ge 2} \eta \left(\frac{t}{2^{-k}}   \right) 
\]
has support in $[0,1/2]$, equals 1 in $[0,1/4]$ and is $C^{\infty}$ in $[1/4,1/2]$. Therefore, we have
\begin{equation}
\label{linear}
 {\bf 1\!\!\!1}_{[1,2]}(t)= \psi(t-1)+\psi(2-t)+\psit(t)
\end{equation}
with $\psit$ some $C^{\infty}$ function with support included in $[1,2]$.

\medskip

In order to get a formula for $F_{s,N}^{{\bf 1\!\!\!1}_{[1,2]}}(x)$, we are going to use  \eqref{linear} and the linearity in $\psi$ of the formula \eqref{defFpsi}.

\medskip
 
 We will first get a formula for $F_{s,N}^{\phi}$, where $\phi$ is any $C^\infty$ function supported in $[1/2,2]$. In particular, this will work with $\phi=\widetilde\psi$. 
 \begin{lemma}
 \label{lemma37}
 Let $\phi$ be a $C^\infty$ function supported in $[1/2,2]$. Then, 
 \begin{equation}
\label{eq10}
 F_{s,N}^{\phi} \left(\frac pq+h \right) = \frac{ N^{1-s} \theta_0 }   {\sqrt q}\displaystyle \int_\R \frac{e^{2i\pi {N^2 h} t^2}}{t^s}\phi(t)dt + G_{s,N}^{\phi}(h) +  \mathcal{O}_{\phi}\left (  \frac{ {q}}{N^{1/2+s}}\right),
\end{equation}
where 
\begin{equation}
\label{defGN}
G_{s,N}^{\phi}(h)= (2hq)^{s-\frac 12} e^{i\pi/4} \sum_{m\neq 0} \frac{\theta_m}{m^s} e^{- i\pi \frac{m^2}{2q^2 h}} \phi \left(\frac{m}{2Nhq} \right).
\end{equation}
 \end{lemma}
 
 \begin{proof}
 
 First,  by \eqref{eq2} one has $ 
 F_{s,N}^{\phi}(x)=N^{1-s} E_N^{\phi_s}(x) $. Further, by (\ref{poisson}) one has
$$
 E_N^{\phi_s} \left(\frac pq+h \right)=  \frac{1}{\sqrt q} \sum_{m \in \Z} \theta_m  \cdot \widehat {w_{N^2 h}^{\phi_s}} \left (\frac{Nm}{q} \right),
$$
and then, applying   Lemma \ref{fourier} with $\xi = \frac{Nm}{q}$ and $R=N^2 h$, one gets
\begin{eqnarray*}
 E_N^{\phi_s} \left(\frac pq+h \right) \!\!& =  &  \theta_0   \frac{\hat w_{N^2 h} ^{\phi_s}(0)}{\sqrt q}    \\
 && +\frac{e^{i\pi/4}}{\sqrt q} \sum_{m\neq 0} \left( \theta_m  \phi_s \left(\frac{m}{2Nqh}\right)  \frac{e^{-i
 \pi \frac{m^2}{2q^2 h}}}{\sqrt{2N^2 h}} \right.
 \\ && \ \ \ \ \ \ \ \ \ \ \ \ \ \ \ \  \left.+ \mathcal{O}\left( { ( Nm/q )^{-3/2} } \right)  \right) .
\end{eqnarray*}

When $ \frac{\xi}{2R} = \frac{ m}{2Nq h}>2$,  $\phi_s \left(\frac{m}{2Nqh}\right)=0$. Recalling that $N \geq q$,  since $\phi_s(t) = \phi(t)t^{-s}$, the above equation can be rewritten
\begin{eqnarray}
\label{eq22}
 F_{s,N}^{\phi} \left(\frac pq+h \right)  & =  &   \frac{ N^{1-s} \theta_0 }   {\sqrt q}\displaystyle \int_\R \frac{e^{2i\pi {N^2 h} t^2}}{t^s}\phi(t)dt   +G_N^\phi (h) \\
\nonumber &&+  \frac{N^{1-s} }{\sqrt{q}} \sum_{m \geq 1}   \mathcal{O}_\phi\left({(Nm/q )^{-3/2}} \right)  . 
\end{eqnarray}
The last term is controlled by
$$ \frac{N^{1-s} }{\sqrt{q}}   \left(\frac 1{(N/q )^{3/2}} \right)    = \frac{{q}}{N^{1/2+s}},$$
which yields  \eqref{eq10}.
\end{proof}

Now, one wants to obtain a comparable formula for  $\eta^k (t) :=\eta((t-1)/2^{-k})$ for all  $k\geq 1$.   We begin with a bound which is good just for large $k$.
 
 \begin{lemma}
 \label{lemma3.8}   For any $k\ge 1$ and $0<h\le 1/q$, one has 
$$
 F_{s,N}^{\eta^k}\left(\frac pq+h\right)=  \frac{   \theta_0 N^{1-s}   }   {\sqrt q}  \displaystyle \int_\R \frac{e^{2i\pi N^2 {  h} t^2}}{ t^s}\eta^k(t)dt+ G_{s,N}^{\eta^k}(h)+     \mathcal{O}_{\eta} \left( N^{1-s}2^{-k}  \right).
$$
\end{lemma}
\begin{proof}
 First,  when $k$ becomes large, since $\eta$ has support in $[1/2,2]$, one has directly:
 -  by \eqref{defFpsi}:
$$
 |F_{s,N}^{\eta^k}(x)| \leq   \sum_{n=1}^{+\infty}  \frac{1}{n^s}  \left|\eta  \left( \frac{\frac{n}{N}  -1}{2^{-k}}\right ) \right|\leq   \sum_{n=N +N2^{-k-1}}^{ N+ N 2^{ -k+1}}  \frac{1}{n^s}  \ll  N^{1-s} 2^{-k } 
$$
- by \eqref{defGN}:
 \[
 |G_{s,N}^{\eta^k}(h)   | \ll (qh)^{s-\frac 12} \sum _{m: \,\eta^k(\frac{m}{2Nhq})\neq 0 } \frac{1}{m^s}   \ll  (qh)^{s-\frac 12} \frac{2Nhq2^{-k}}{(2Nhq)^{s}}    \ll  \sqrt{qh} 2^{-k} N^{1-s}
\]
- and
$$ \left|\frac{\theta_0 N^{1-s} }   {\sqrt q}   \displaystyle \int_\R \frac{e^{2i\pi N^2 {  h} t^2}}{ t^s}\eta^k(t)dt  \right| \ll N^{1-s}q^{-1/2}     \int_{1+2^{-k-1}} ^{1+2^{-k+1}}  \frac{dt}{ t^s}  \ll 2^{-k}  N^{1-s}q^{-1/2}  ,$$
hence the result by \eqref{eq22}, where we used that $qh \leq 1$.
\end{proof}
   
One can obtain another bound  that is good for any $k$. 
 \begin{lemma}
 \label{lemma3.9}
 For every $k\geq 2$ and $0<h\le 1$, one has
$$
 F_{s,N}^{\eta^{k}} \left(\frac pq+h \right)     =  \frac{ \theta_0 N^{1-s}}   {\sqrt q}   \displaystyle \int_\R \frac{e^{2i\pi N^2 {  h} t^2}}{ t^s}  {  {\eta^k}} (t)dt  +G_{s,N}^{{\eta^k}} (h) +   \mathcal{O}_\eta \left (   \frac{\sqrt{q}}{N^s} +  N^{1/2-s}\gamma_k  \right),$$
 {where the sequence  $(\gamma_k)_{k\geq 1} $ is positive and satisfies $\sum_{k\ge 1} \gamma_k\ll 1$.}
 \end{lemma}
 
 \begin{proof} The proof starts as the one of Lemma \ref{lemma37}. 
 Using the fact that for any  $ (\eta^k)_s  (t)=  \frac{\eta((t-1)/2^{-k})}{t^s} $ and that 
\begin{eqnarray*}
 \widehat {w_R^{(\eta^{k})_s}} (\xi) & =  &  \int_\R   (\eta^{k})_s  (t) e^{2i\pi(Rt^2-  t \xi)} dt= \int_\R    \frac{\eta (\frac{t-1}{2^{-k}})}{t^s}   e^{2i\pi(Rt^2-  t \xi)}    dt \\
& = &    2^{-k}  e^{2i\pi (  R-\xi )}    \int_\R \frac{\eta (u  )}{ (1+u2^{-k})^s}  e^{2i\pi(R2^{-2k} u   ^2-   2^{-k} ( \xi -2R   )  u )} du\\ 
 &  = &      2^{-k}     e^{2i\pi (R-\xi)}  \widehat {w_{ R 2^{-2k}}^{\widetilde {\eta^k}}} (2^{-k}(\xi-2R))  ,
\end{eqnarray*}
where $\widetilde {\eta^k} (u)= \frac{\eta (u  )}{ (1+u2^{-k})^s} $. Hence, 
\begin{eqnarray}
 \nonumber E_N^{(\eta^k)_s} \left(\frac pq+h \right)  & = &  \frac{1}{\sqrt q} \sum_{m \in \Z} \theta_m  \cdot \widehat {w_{N^2 h}^{(\eta^k)_s}} \left (\frac{Nm}{q} \right)\\
\label{eqplustard}   & = &  \frac{ 2^{-k}    }{\sqrt q} \sum_{m \in \Z} \theta_m  \cdot   e^{  2i\pi(N^2h-  \frac{Nm}{q } )}    \widehat {w_{ N^2h 2^{-2k}}^{\widetilde {\eta^k}}} \left (2^{-k} \left (\frac{Nm}{q}-2N^2h  \right) \right)  
  \end{eqnarray}
  Here we apply again Lemma \ref{fourier}
 and we obtain
  \begin{eqnarray*}
  E_N^{(\eta^k)_s} \left(\frac pq+h \right)     &=&   \frac{ 2^{-k}    }{\sqrt q}  \sum_{m\neq 0}  \theta_m e^{  2i\pi(N^2h-  \frac{Nm}{q } )}   \\
 && \times   \left(  e^{i\pi/4} {\widetilde {\eta^k}} \left(\frac{2^{-k} \left (\frac{Nm}{q}-2N^2h  \right) }{2N^2h 2^{-2k}}\right)  \frac{e^{-i  \pi \frac{2^{-2k} \left(\frac{Nm}{q}  - 2N^2h \right )  ^2}{2 N^2 h2^{-2k}}}}{\sqrt{2N^2 h 2^{-2k}}}  \right.\\ 
 && \left.+ \mathcal{O}_{\widetilde {\eta^k}} \left({\frac{\rho_{R_k,\xi_k}}{\sqrt{2^{-2k}N^2 h}}+} { \left( 1+ 2^{-k} \left|\frac{Nm}{q}  - 2N^2h \right |+ N^2h2^{-2k}\right) ^{-3/2}} \right) \right) .
\end{eqnarray*}
{with $R_k=2^{-2k}N^2 h$ and $\xi_k=2^{-k}|Nm/q-2N^2h|$.}
Finally, after simplification, one gets
\begin{eqnarray}
\nonumber F_{s,N}^{\eta^{k}} \left(\frac pq+h \right)  & =  & N^{1-s} E_N^{(\eta^k)_s} \left(\frac pq+h \right) \\
& =  &   \frac{(2qh)^{s-\frac 12}}{e^{-i\pi/4}} 
 \sum_{m \in \Z}         \frac { \theta_m  }  {  m ^s} e^{  - 2i\pi  \frac{m^2}{q^2h}   }   {  {\eta^k}}    \left(  \frac{m}{2Nhq}   \right)  \\
 && \!\!\!\! \!\!\!\!\!\!   \\
\nonumber  & =  &  \frac{  \theta_0 N^{1-s} }   {\sqrt q} \displaystyle \int_\R \frac{e^{2i\pi N^2 {  h} t^2}}{ t^s}  {  {\eta^k}} (t)dt  +G_{s,N}^{{\eta^k}} (h) +   L^k_N , 
 \end{eqnarray}
{
where by Lemmas \ref{lemma37} and \ref{lemma3.8} one has
\[
L^k_N =\mathcal{O} _{\widetilde{\eta^k}} \left( \sum_{m\in J^k_N\cap \mathbb Z^*}  \frac{2^{-k}N^{1-s}/\sqrt q}{\sqrt{2^{-2k} N^2 h} }  +\sum_{m \in \Z^*}  \frac{2^{-k}N^{1-s}/\sqrt q}{(1+ 2^{-k} |\frac{Nm}{q}  - 2N^2h|+ N^2h2^{-2k})^{3/2} }\right)\!\!,
\]
with $J^k_N= \big [(2+2^{-k-1})Nqh,(2+2^{-k+1})Nqh \big]$.

\smallskip

First, as specified in Lemma \ref{fourier}, the constants involved in the $ \mathcal{O}_{\widetilde {\eta^k}}$ depend on upper bounds for some derivatives of $\widetilde {\eta^k}$, and then by the definition of $\widetilde{\eta^k}$ we can assume they are fixed and independent on both $k$ and $s$.

\medskip
Let $\{x\}$ stand for the distance from the real number $x$ to the nearest integer.

The first sum in $L^k_N$ is bounded above by:
\begin{itemize}
\smallskip\item
 $\sqrt q N^{-s} +  N^{1/2-s}$ when  $2^{-k}\in\big [\{2Nhq\}/4Nhq,\{2Nhq\}/Nhq \big]$, 
\smallskip \item  $\sqrt q N^{-s} $ otherwise.
\end{itemize}

In particular, $x$ being fixed, the term $N^{1/2-s}$ may appear only a finite number of times when $k$ ranges in $\N$. 

In the second sum, there is at most one integer $m$ for which $|Nm/q-2N^2h|< N/2q$, and the corresponding term is bounded above by 
\begin{eqnarray*}
&&2^{-k} N^{1-s}q^{-1/2}(1+N/q 2^{-k}+ | N^2 h| ^{-2k})^{-3/2} \\
&\leq & 2^{-k} N^{1-s}q^{-1/2}(1+ N/q 2^{-2k})^{-3/2} \\
 & = & N^{1/2-s}  \frac{ N^{1/2} q^{-1/2} 2^{-k} }{(1+ N/q 2^{-2k})^{-3/2}}\\
& \leq &   N^{1/2-s} \gamma_k, 
\end{eqnarray*}
 where $\gamma_k= \sqrt{ \frac{u_k}{(1+u_k)^{3}}}$ and $u_k= 2^{-2k}N/q$. The sum over $k$ of this upper bound is finite, and  this sum can be bounded above independently on $N$ and $q$.
 
 \smallskip
 
 The rest of the sum is bounded, up to a multiplicative constant, by
 $$ \int_{u=0}^{+\infty}  \frac {\sqrt{q}N^{-s}\,(2^{-k}N/q) du}{ (1 + N^2h 2^{-2k} + 2^{-k} |\frac{Nu}{q}  - 2N^2h  | )^{3/2}}  \ll \frac{\sqrt q N^{-s}}{(1+N^2 h 2^{-2k})^{1/2}}\ll 
 \sqrt q N^{-s},$$
 hence the result.}
 \end{proof}

 Now we are ready to prove Proposition \ref{summation}.
 
 \medskip

Recall that $N\geq q$ and $0\leq h\leq q^{-1}$.  Let $K$ be the unique integer such that $2^{-K }\leq \frac{\sqrt{q}}{N}< 2^{-(K+1)}$.  We need to bound by above the sum of the errors $L^k_N$:
 \begin{itemize}
 
 \item when $k\geq K$: we use Lemma \ref{lemma3.8} to get
 $$ \sum_{ k\geq K} |L^k_N| \ll   N^{1-s} 2^K \ll  N^{1-s} \frac{\sqrt{q}}{N}  =\frac{\sqrt{q}}{N^s} \leq \frac{1}{N^{s-1/2}}. $$

 \item the remaining terms are simply bounded using Lemma \ref{lemma3.9} by
  $$ \sum_{ k=2}^{K} |L^k_N| \ll K \frac{\sqrt{q}}{N^s} {  + N^{1/2-s} } \ll    \log N   \frac{\sqrt{q}}{N^s} +N^{1/2-s}\ll \log q \frac{\sqrt{N}}{N^s} = \frac{\log q}{N^{s-1/2}},$$
where we use that the mapping $x\mapsto \frac{\sqrt{x}}{\log x}$ is increasing for large $x$.
   \end{itemize}
 
 Gathering all the informations, and  recalling that $\sum_{k=2}^{+\infty}  {\eta^k}   (t) = \psi(t-1)$,  we have that 
 $$  F^{\psi(\cdot -1) }_{s,N} \left(\frac{p}{q}+h\right) = \frac{ N^{1-s} \theta_0  }   {\sqrt q}     \int_\R \frac{e^{2i\pi {N^2 h} t^2}}{t^s} \psi(t -1) dt+   G_{s,N}^{\psi(\cdot -1)} (h) +  \mathcal{O}_{\eta}\left (  \frac{ {\log q}}{N^{{s-1/2}}}\right).$$

   The same inequalities remain true if we use the functions  $ {\widetilde {\eta^k} } = \eta((2-t)/2^{-k})$, so the last inequality also holds for $\psi(2-\cdot) $
   
   \medskip
   
   Finally, recalling the decomposition \eqref{linear} expressing  ${\bf 1 \!\!\!1}_{[1,2]}$ in terms of smooth functions, we get 
   \begin{equation}
   \label{eqFN}
     F^ {{\bf 1\!\!\!1}_{[1,2]}}_{s,N} \left(\frac{p}{q}+h\right) = \frac{ N^{1-s} \theta_0  }   {\sqrt q}     \int_1^2 \frac{e^{2i\pi {N^2 h} t^2}}{t^s} dt+   G_{s,N}^ {{\bf 1\!\!\!1}_{[1,2]}} (h) +  \mathcal{O}_{\eta}\left (  \frac{ {\log q}}{N^{{s-1/2}}}\right),
   \end{equation}
   and the result follows.
\end{proof}

\section{Proof of the convergence theorem \ref{convergence}}

\subsection{Convergence part: item (i)}
Let $x$ be such that \eqref{S<+infty} holds true.

Recall the definition \eqref{defconvergents} of the convergents of $x$. 
We begin by bounding $F_{s,M}(x)-F_{s,N}(x)$ for any 
\[
 q_j/{4}  \le N < M < q_{j+1}/4.
\]
We apply Proposition \ref{summation} with $p/q=p_j/q_j$ and $h=h_j$, so that $x=p/q+h$. {Due to (\ref{symmetry}), we can assume that $h_j>0$}. It is known that for $ \frac{1}{2q_jq_{j+1}}  \leq h_j  = |x- p_j/q_j| <  \frac{1}{q_jq_{j+1}}   $. 

First, since $4Nh_jq_j< 4N/q_{j+1}<1$, the sums \eqref{DEFGN} appearing in 
$  G_{s,2N}(h_j)$ and $G_{s,N}(h_j )$ have no terms, hence are equal to zero. This yields
\[
 F_{s,2N}(x)-F_{s,N}(x)= \frac{\theta_0}{\sqrt{q_j}} \ \int_N^{2N}\frac{e^{2i\pi h_j t^2}}{t^s}dt + O(N^{\frac 12 -s} \log q_j).
\]
It is immediate to check that  $\int_a^{2a} t^{-s} e^{2i\pi t^2}\,dt \ll  \min(a^{-s-1}, a^{-s+1})$, thus
\begin{eqnarray*}
 \left| \int_N^{2N}\frac{e^{2i\pi h_j t^2}}{t^s}dt\right|  & \ll &   |h_j|^{s/2-1/2} \left| \int_{N\sqrt{h_j} } ^{2N\sqrt{h_j} }\frac{e^{2i\pi u^2}}{u^s}du  \right|\\
 &\ll &   |h_j|^{ -1/2}   N  ^{-s } \min (|N \sqrt{h_j} |^{- 1}, |N \sqrt{h_j} | )  .
 \end{eqnarray*}

One deduces (using that $q_jh_j $ is equivalent to $q_{j+1}^{-1}$)  that
\[
 |F_{s,2N}(x)-F_{s,N}(x)|\ll |\theta_0|     \frac{\sqrt{q_{j+1}}}{N^s} \min (\frac{N}{\sqrt{q_j q_{j+1}}}, \frac{\sqrt{q_j q_{j+1}}}{N})+ N^{\frac 12-s} \log q_j.
\]
{Thus, by writing $F_{s,M}(x)-F_{s,N}(x)$ as a dyadic sum we have
\[
 |F_{s,M}(x)-F_{s,N}(x)| \ll |\theta_0|\delta_j\frac{\sqrt{q_{j+1}}}{(\sqrt{q_j q_{j+1}})^s} + \frac{\log q_j}{q_j^{s-1/2}}.
\]
Recalling that $\theta_0$ is equal to zero when $q_j\neq 2\times $odd, fixing an integer $j_0 \geq 1$, for any $M>N>q_{j_0}$, one has }

\[
| F_{s,M}(x)-F_{s,N}(x)| \ll \sum_{j\ge j_0, q_j\neq 2*\odd} \delta_j\frac{\sqrt{q_{j+1}}}{(\sqrt{q_j q_{j+1}})^s} + \sum_{j\ge j_0} \frac{\log q_j}{q_j^{s-1/2}} +\sum_{j\geq j_0} \frac{1}{q_{j+1}^{s-1/2}}.
\]
The second and third series always converge when $j_0\to \infty$, and the first does when $\Sigma_s (x)<\infty$.

\subsection{Divergence part: item (ii)}
 
Let $0<\ep<1/2$ a small constant. Let $N_j=\ep q_j$ and $M_j= 2\ep \sqrt{q_j q_{j+1}}$.  Proceeding exactly as in the previous proof we get
\[
 F_{s,M_j}(x)-F_{s,N_j}(x) = \frac{\theta_0}{\sqrt{q_j}}     \int_{N_j}^{M_j}\frac{e^{2i\pi h_j t^2}}{t^s}dt  + \mathcal{O}\left( q_j^{\frac 12 -s} \log q_j  \right).
\]
Since $e^{2i\pi h_j t^2} = 1+\mathcal{O}(\ep)$ inside the integral, as soon as $q_j\neq 2*\odd$, one has
 \[
 |F_{s,M_j}(x)-F_{s,N_j}(x)| \geq \frac{|\theta_0| }{\sqrt{q_j}} \frac{M_j-N_j}{2 \cdot M_j^s} \geq |\theta_0|  \ep   \frac{2\sqrt{q_{j+1}}-\sqrt{q_j}}{2 ^{1+s} \cdot  \ep^s \cdot (q_jq_{j+1})^{s/2}}  \gg \sqrt{\frac{q_{j+1}}{(q_jq_{j+1})^s}},
\] 
which is infinitely often large by our assumption. Hence the divergence of the series.

\section{Local $L^2$ bounds for the function $F_s$}

%

Further intermediary results are needed to study  the  local regularity  of $F_s$.

\begin{proposition}\label{summation_substract}
Let $h>0$, $1/2<s<3/2$ and $q^2 h\ll 1$. We have
 \begin{eqnarray}
  \label{eqFsN}
  \ \ \ F_{s,N}\left(\frac pq+2h\right )-F_{s,N}\left(\frac pq+h\right) & =  & \frac{\theta_0}{\sqrt{q}}  \int_0^N \frac{e^{2i\pi 2ht^2}-e^{2i\pi ht^2}}{t^s}dt\\
  \nonumber& &  + \, G_{s,N}(2h)-G_{s,N}(h)\\
  \nonumber &&  + {  \mathcal{O} \left( |qh|^{s-1/2}  \right ). }
 \end{eqnarray}
\end{proposition}
\begin{proof}
First,  one writes  
\begin{equation}
\label{decomp}
 F_{s,N}(x) = F_{s,N}^{{\bf 1\!\!\!1}_{[0,1]}} (x)=\sum_{m\ge 1}  F_{s,  N/2^m  }^{{\bf 1\!\!\!1}_{[1,2]}}(x) .
\end{equation}

  Observe that when $N$ is divisible by 2, there may be some terms appearing twice in the preceding sum, so there is not exactly equality. Nevertheless, in this case, only a few terms are added and they do not change our  estimates. This is left to the reader.
  
  \medskip

We are going to estimate    \eqref{eqFsN} but  with $ F_{s,N}^{{\bf 1\!\!\!1}_{[1,2]}}$ and $ G_{s,N}^{{\bf 1\!\!\!1}_{[1,2]}}$ instead of $ F_{s,N} $ and $ G_{s,N} $, with an error term suitably bounded by above.
Then, using this result with $N$ substituted by $N/2^m$, and then summing over $m = 1,..., \lfloor\log_2N\rfloor$ will give the result (for $m> \lfloor \log _2 N \rfloor$, the sum $F_{s,N/2^m}^{{\bf 1\!\!\!1}_{[1,2]}}$ is empty).

 \medskip

 We start from equation \eqref{eqplustard} applied with $h$ and $2h$, {and then we apply Lemma \ref{fourier}, but this time equation (\ref{fourier_substract}) instead of (\ref{eq3})}. Let us introduce for all integers $k$  the quantity
 \begin{eqnarray}
 \label{eq33} 
 E^k_N & := &  F_{s,N}^{\eta^{k}} \left(\frac pq+2h \right)  - F_{s,N}^{\eta^{k}} \left(\frac pq+h \right)   \\ 
  \nonumber
  &&  \!\!\!\!\!\!+ \,  \frac{ \theta_0 N^{1-s}}   {\sqrt q}    \! \int_\R \frac{e^{2i\pi N^2 {  2h} t^2}-e^{2i\pi N^2 {  h} t^2}}{ t^s}  {  {\eta^k}} (t)dt \\
 \nonumber&&   \!\!\!\!\!\!+ \, G_{s,N}^{{\eta^k}} (2h) -G_{s,N}^{{\eta^k}} (h)  ,
 \end{eqnarray} 
with $\eta^k$ defined as in Proposition \ref{summation}. By the exact same computations as in Lemma \ref{lemma3.9}, one obtains
 {the upper bound 
\begin{eqnarray*}
|E^k_N| \ll \beta^k_N\sum_{m\in J_k\cap \mathbb Z^*}  \frac{2^{-k}N^{1-s}/\sqrt q}{(2^{-2k} N^2 h)^{-1/2} } +  \sum_{m \in \Z^*}  \frac{(2^{-k}N^{1-s}/\sqrt q)N^2 h 2^{-2k}}{(1 + N^2h 2^{-2k} + 2^{-k} |\frac{Nm}{q}  - 2N^2h| )^{5/2}}, \end{eqnarray*}
with $J^k_N= \big[(2+2^{-k-1})Nqh,(2+2^{-k+1})Nqh \big]$ and 
$$\beta^k_N =\ \begin{cases} \  1 &\mbox{  if }  \ 2^{-2k}N^2 h\le 1,\\  \ 0& \mbox{ otherwise.}\end{cases}$$
 Then, as at the end of the proof of Lemma \ref{lemma3.9}, since $h\ll q^{-2}$, we can bound the sums by
\[
 |E^k_N|\ll \beta^k_N 2^{-2k}N^{2-s}\frac{\sqrt h}{\sqrt q}+ \frac{ (N^{1/2-s}) \sqrt{(N/q)2^{-2k}} N^2 h 2^{-2k}}{(1+N^2 h 2^{-2k}+(N/q)2^{-2k})^{5/2}}+ \frac{(\sqrt q N^{-s})N^2h2^{-2k}}{(1+N^2 h 2^{-2k})^{3/2}} 
\]
and then adding up in $k\ge 1$ we get
\[
 \sum_{k=1}^{\infty} |E^k_N |\ll \widetilde E_N=\frac{ \sqrt{1/hq}}{N^{s}}\min(1,Nqh)+\frac{ \sqrt{N}}{N^{s}}\min(1,Nqh) + \frac{\sqrt q}{N^s} \min(1,N^2h).
\]
The same holds true for the functions     $ {\widetilde {\eta^k} } = \eta((2-t)/2^{-k})$, and for $\widetilde{\psi}$ since it is similar to $\eta^1$, so by (\ref{linear}) we finally obtain that }
 \begin{eqnarray}
  \label{eqFsN2}
  \ \ \ F_{s,N}^{{\bf 1\!\!\!1}_{[1,2]}}  \left(\frac pq+2h\right )     -F_{s,N}^{{\bf 1\!\!\!1}_{[1,2]}} \left(\frac pq+h\right) & =  & \frac{\theta_0}{\sqrt{q}}  \int_0^N \frac{e^{2i\pi 2ht^2}-e^{2i\pi ht^2}}{t^s}dt\\
  \nonumber& &  + \, G_{s,N}^{{\bf 1\!\!\!1}_{[1,2]}}(2h)-G_{s,N}^{{\bf 1\!\!\!1}_{[1,2]}}(h)\\
  \nonumber &&  + {\mathcal{O} \left( \widetilde E_N\right)}.
 \end{eqnarray}

The same holds true with $N/2^m$ instead of $N$. To get the result, using \eqref{decomp}, it is now enough to sum the last inequality  over $m = 1,...,  \lfloor \log _2 N \rfloor$. {Let us treat the first term. One has
 \begin{eqnarray*}
  \sum_{m = 1}^{  \lfloor \log _2 N \rfloor } \frac{ \sqrt{1/hq}}{N^{s}}\min(1,Nqh)   & = &  \sum_{m = 1}^{  \lfloor \log _2 N \rfloor } \frac{ \sqrt{1/hq}}{(N2^m)^{s}}\min(1,N2^mqh)  \\
  & = &  \frac{ \sqrt{1/hq}}{ N ^{s}}  \sum_{m = 1}^{  \lfloor \log _2 N \rfloor }\min(2^{ms},N2^{m(1-s)}qh)  \\
 & \leq &  \frac{ \sqrt{1/hq}}{ N ^{s}}  \sum_{m = 1}^{ +\infty }\min(2^{ms},N2^{m(1-s)}qh)  \\
 & \ll & \frac{ \sqrt{1/hq}}{ N ^{s}} 2^{Ms},
   \end{eqnarray*}
   where $M$ is the integer part of the solution of the equation $2^{Ms} = N2^{m(1-s)}qh$, i.e. $2^{M} \eqsim Nqh$. Hence the first sum is bounded above by $\frac{\sqrt{1/qh}}{(1/qh)^s}$. The other terms are treated similarly, and  finally} \eqref{eqFsN} is true with an error term bounded by above by  \[
\mathcal{O} \left (\frac{\sqrt{1/qh}}{(1/qh)^s}+ \frac{\sqrt{q}}{(h^{-1/2})^{-s}} \right)
\]
which is  $\mathcal{O} ((qh)^{s-1/2})$ on $h\ll q^{-2}$. 
\end{proof}

We also need to control the $L^2$ norm of the main term.

\begin{lemma}\label{L2_main}
Let $0<s\le 1$ and fix  {$0<H<1$}. Let 
$$f_{s,N}(\cdot)=     \int_0^{N} \frac{e^{2i\pi t^2(\delta + 2 \cdot)} -e^{2i\pi t^2(\delta +\cdot)}  }{t^s}dt.$$
Then for any $N>0$ , $\left\| f_{s,N}(\cdot)  \right\|_{L^2(\widetilde{\mu}_H)} \ll \min\left( H^{(s-1)/2}, H |\delta|^{(s-3)/2} \right).$
\end{lemma}

\begin{proof}
$\bullet$ Let us treat first the case $|\delta|<H/4$. Using a change of variable, one has
\[
f_{s,N}(h)=  H^{(s-1)/2} \int_0^{N\sqrt H} \frac{e^{2i\pi t^2\frac{\delta+2h}{H}}-e^{2i\pi   t^2\frac{\delta +h}{H}}}{t^s} \, dt.
\]
We are interested in the range  $H<h<2H$, and in this case the ratios    $\frac{\delta+2h}H$, $\frac{\delta+h}{H}$ are bounded, so that  the integral is  bounded by a constant independent of $N$. One deduces that   $\left\| f_{s,N}(\cdot)  \right\|_{L^2(\widetilde{\mu}_H)} \ll H^{(s-1)/2}$.

\medskip

$\bullet$ Assume then that $|\delta|>4H$. Assume that $\delta>0$ (the same holds true with negative $\delta$'s). Using a change of variable, one has
\[
f_{s,N}(h)= |\delta|^{(s-1)/2} \int_0^{N\sqrt{|\delta|}} \frac{ e^{2i\pi t^2 (1+\frac{2h}{\delta})}-e^{2i\pi t^2 (1+\frac{h}{\delta})}}{t^s} \,dt.
\]
The integral between 0 and 1 is clearly $\mathcal{O}(h/|\delta|)$. For the other part, one has (after integration by parts)
\begin{eqnarray*}
 \int_1^{N\sqrt{\delta}} \frac{ e^{2i\pi t^2 (1+\frac{2h}{\delta})}-e^{2i\pi t^2 (1+\frac{h}{\delta})}}{t^s} \,dt   &=& \mathcal{O} \left (    {h} / {|\delta|} \right) ,\end{eqnarray*}
so that  $|f_{s,N}(h)|\ll  H   |\delta|^{(s-3)/2}$ for any $H<h<2H$. Hence $\left\| f_{s,N}(\cdot)  \right\|_{L^2(\widetilde{\mu}_H)} \ll  H   |\delta|^{(s-3)/2}$.

\medskip

$\bullet$ It remains us to deal with the case $H/4<\delta\leq 4H$. One observes that 
$$
f_{s,N}(h)= D(\delta+2h)-D(\delta+h) + \mathcal{O}(H), \ \ \mbox{ where  } \ D(v)=\int_1^N t^{-s} e^{2i\pi vt^2} \, dt.
$$ 
 It is enough to get the bound
\[
 \int_0^H |D(v )|^2 \,dv \ll   H^s{\color{red},}
\]
which follows from the fact that  $|D(v)| \ll |v|^{(s-1)/2}$ when $s<1$ and $|D(v)| \ll 1+\log(1/|v|)$ when $s=1$.

\end{proof}

Finally, the oscillating behavior of $G_{s,N}(h)$  gives us the following.

\begin{proposition}\label{L2_average}
Let $0<H\leq q^{-2}$ and $|\delta|\leq \sqrt{H}/q$. Let 
$$g_{s,N}(\cdot)=   F_{s,N}\left(\frac pq+\delta +2\,\cdot \right)-F_{s,N} \left(\frac pq+\delta+\cdot\right)- \frac{ \theta_0}{\sqrt{q}} f_{s,N}(\cdot).$$
One has $\|g_{s,N}\|_{L^2(\widetilde \mu_H)} \ll  { H^{\frac{s-1/2}2} }. 
$
\end{proposition}

\begin{proof}
We consider $\mu_H = (\widetilde \mu_H)_{|\R^+}$. 
 By (\ref{symmetry}), it is enough to treat the case $\delta+h>0$ and $\delta+2h>0$. Proposition \ref{summation_substract}, applied successively with $h_n:=2^{-n}(\delta +{h})$ and $\widetilde h_n:=2^{-n}(\delta +{h/2})$, and summing over $n\geq 0$, we get that
 \begin{eqnarray}
  \label{eqFsN3}
g_{s,N}(h)  = G_{s,N}(\delta+2h)-G_{s,N}(\delta+h)   + O\left( {(q(|\delta|+|h))^{s-1/2} } \right ).
 \end{eqnarray}
Thus, {since $q(|\delta|+|h|)\ll \sqrt{H}$,} it is enough to show that
\begin{equation}
\label{delta_bound}
 \|G_{s,N}(\delta + \cdot)\|_{L^2(\mu_H)} \ll H^{\frac{s-1/2}2}  .
 \end{equation}
Assume first that $|\delta|\geq 3H$. By expanding the square and changing the order of summation, and using that $\delta+2H \leq 2|\delta|$,  we have  for some  $c_{n,m}\ge 0$ 
\begin{eqnarray*}
\left \|G_{s,N}(\delta+\cdot)\right\|^2_{L^2(\mu_H)} \!\!\!   &  \ll &  \!\!\! (q|\delta|)^{2s-1}  { \sum_{  n,m=1}^{2 \lfloor2 N|\delta| q \rfloor} \frac{|\theta_m|}{m^s} \frac{|\theta_n|}{n^s} \left  |\int_{\delta+H{ +c_{n,m}} }^{\delta+2H} e^{2i\pi \frac{n^2-m^2}{q^2 h}}      \, \frac{dh}{H}   \right|} \\
 &  \ll &  (q|\delta|)^{2s-1}   \sum_{  n,m=1}^{ { 2}\lfloor2 N|\delta| q \rfloor} \frac{|\theta_m|}{m^s} \frac{|\theta_n|}{n^s} \left  |\int_{\delta+H{} }^{\delta+2H}  {e^{2i\pi \frac{n^2-m^2}{q^2 h}}   \, \frac{dh}{H} }  \right|.
\end{eqnarray*}
 Since for $|M|\geq 1$ and $0<\ep\ll 1$
\[
 \int_1^{1+\ep} e^{2i\pi \frac Mt}\,dt\ll \frac{1}{|M|},
\]
 the previous sum is bounded above by 
\[
{(q|\delta|)^{2s-1}}\left[ \sum_{m\ge 1} \frac 1{m^{2s}} + { \frac{|\delta|}{H} }  q^2 |\delta|\sum_{m\ge 1} \frac{1}{m^{1+s}} \sum_{j\ge 1} \frac 1{j^{1+s}} \right],
\]
with $j=|n-m|$. The term between brackets is bounded by a universal constant (since  {$q^2\delta^2 /H\leq 1$)}, hence \eqref{delta_bound} holds true. It is immediate that the same holds true   with $(\widetilde \mu_H)_{|\R^-}$.

Further,  assume that $  |\delta|<3H$. Setting $H_k=2^{-k}H$,  one has 
\[
 \left\|G_{s,N}(\delta +\cdot) \right\|_{L^2(\widetilde \mu_H)}^2 \leq \int_0^{5H} |G_{s,N}(h)|^2 \frac{dh}{H} \leq \sum_{k\ge -2} 2^{-k} \| G_{s,N}(\cdot) \|_{L^2(\mu_{H_k})}^2.
\]

Now,  observing that $[H_k,H_{k-1}] \subset 3H_k + \left([-2H_k,-H_k]\cup[H_k,2H_k]\right)$, one can apply  \eqref{delta_bound} with  $H=H_k$ and $\delta=3H_k$ to get
$$ \| G_{s,N}(\cdot) \|_{L^2(\mu_{H_k})}^2 \leq  \|G_{s,N}(\delta_k + \cdot)\|_{L^2(\mu_{H_k})} \leq  H_k^{\frac{s-1/2}2} =H^{\frac{s-1/2}2} 2^{-k \frac{s-1/2}2} .$$
Summing over $k$ yields the result.
\end{proof}

\section{Proof of Theorem \ref{smoothness}}

\subsection{Lower bound for the local $L^2$-exponent $\alpha_{F_s}$}

Assume that $\Sigma_s (x)<\infty$ (see equation \eqref{S<+infty}), so that the series $F_{s,N}(x)$ converges to $F_s(x)$. Recall that $p_j/q_j$ stands for the partial quotients of $x$.

\medskip 

Pick $N$ such that $0\leq |F_s(x)-F_{s,N}(x)|<H$ and  $N^{\frac 12-s}\leq H^2$. Since
\begin{eqnarray*}
 \left\|{F_{s}}(x+\cdot)-F_{s,N}(x+\cdot)  \right\|_{L^2(\widetilde\mu_H)} & \le & \frac{\|{F_{s}}(x+\cdot)-F_{s,N}(x+\cdot)\|_{L^2([0,1])}}{H/2} \\
  & \ll & \frac{N^{\frac 12-s}}{H}  \leq H,
\end{eqnarray*}
and  since one has
\[
 {F_{s}}(x+\cdot)-{F_{s}}(x)={F_{s}}(x+\cdot)-F_{s,N}(x+\cdot)+F_{s,N}(x+\cdot)-F_{s,N}(x)+F_{s,N}(x)-{F_{s}}(x),
\]
one deduces that
\[
 \|{F_{s}}(x+\cdot)-{F_{s}}(x)\|_{L^2(\widetilde\mu_H)}=\|F_{s,N}(x+\cdot)-F_{s,N}(x)\|_{L^2(\widetilde\mu_H)}+\mathcal{O}(H).
\]
Thus, it is enough to take care of the local $L^2$-norm of $   F_{s,N}(x+h)-F_{s,N}(x)$. One has 
\begin{eqnarray}
\nonumber
 \|{F_{s,N}(x+h)-F_{s,N}(x)}\|_{L^2(\widetilde\mu_H)}&\le &\sum_{k\ge 1} \left \|F_{s,N}(x+2\frac{\cdot}{2^k})-F_{s,N}(x+\frac{\cdot}{2^k}) \right\|_{L^2(\widetilde\mu_H)}\\
 \label{dyadic_L2}
&\leq & \sum_{k\ge 1} \left \|F_{s,N}(x+2\,\cdot)-F_{s,N}(x+\cdot) \right\|_{L^2(\widetilde\mu_{H_k})},
\end{eqnarray}
where $H_k = H2^{-k}$. Let us introduce the function $f(h) =F_{s,N}(x+2\, h )-F_{s,N}(x+h)$.

Let  $j_H$ be the smallest integer such that $q_j^{-2} \leq H$. For every $k\geq 1$,    and let $j$ be the unique integer such that $q_{j+1}^{-2}\leq H_k < q_j^{-2}$ (necessarily $j\geq j_H-1$). Using that $|x-p_j/q_j|  = |h_j|\leq q_j^{-2}$, one sees that
\[
\left \|f \right\|_{L^2(\widetilde\mu_{H_k})}=  \left  \|F_{s,N} \left  (\frac{p_j}{q_j}+h_j+2\,\cdot \right )-F_{s,N} \left  (\frac{p_j}{q_j}+h_j+\cdot\right ) \right \|_{L^2(\widetilde\mu_{H_k})}.
\]
Since $ | h_j|  <1/q_jq_{j+1}\leq \sqrt{H_k}/q_j$, we can apply Proposition \ref{L2_average} and Lemma \ref{L2_main}  with $H_k$ and $\delta = h_j$ to get
\begin{eqnarray*}
 \|f\|_{L^2(\widetilde\mu_{H_k})} & \ll &  H_k^{\frac{s-1/2}2} +\frac{|\theta_0|}{\sqrt{q_j}}  \min\left( H_k^{(s-1)/2}, H_k |h_j|^{(s-3)/2} \right)\\
 & \ll &  H_k^{\frac{s-1/2}2} +\frac{|\theta_0|}{\sqrt{q_j}} H_k^{(s-1)/2}  \min\left( 1,  \left| \frac{h_j}{H_k}\right|^{(s-3)/2} \right).
 \end{eqnarray*}

In order to finish the proof we are going to consider three different cases: 
 
 \medskip
 
 {\bf (1)  $s-1+1/2r_{\odd}(x)>0$:}  Since $h_j=q_j^{-r_j}$ we have
 \[
  \|f\|_{L^2(\widetilde\mu_{H_k})} \ll  H_k^{\frac{s-1/2}2} +|\theta_0| H_k^{(s-1)/2}  \min\left( |h_j|^{\frac 1{2r_j}},   \frac{{H_k}^{(3-s)/2}}{|h_j|^{\frac{3-s}2-\frac 1{2r_j}}} \right),	
 \]
and optimizing in $|h_j|$ we get
\[
  \|f\|_{L^2(\widetilde\mu_{H_k})} \ll  H_k^{\frac{s-1/2}2} +|\theta_0| H_k^{\frac{s-1+1/2r_j}2}\ll H_k^{(s-1+1/2r_{\odd}(x)+o(H_k))/2} 
\]
by the definition of $r_{\odd}(x)$. Adding up in $k$ finishes the proof in this case.
 
\medskip

 {\bf (2)  $s-1+1/2r_{\odd}(x)=0$ and $s=1$:} In this case it is enough to show that $\sum_{k\ge 1} \|f\|_{L^2(\widetilde\mu_{H_k})}<\infty$. We have 
 \[
 \|f\|_{L^2(\widetilde\mu_{H_k})} \ll  H_k^{\frac{s-1/2}2}+ \frac{|\theta_0|}{\sqrt{q_j}}
 \]
which implies
\[
 \sum_{q_{j+1}^{-2}\le H_k \le q_j^{-2}}\|f\|_{L^2(\widetilde\mu_{H_k})} \ll  \sum_{q_{j+1}^{-2}\le H_k \le q_j^{-2}} H_k^{\frac{s-1/2}2}  + \frac{|\theta_0|}{\sqrt{q_j}}\log(q_{j+1}/q_j).
\]  
This yields
\[
 \sum_{k\ge 1}\|f\|_{L^2(\widetilde\mu_{H_k})}\ll H^{\frac{s-1/2}2} + \sum_{j: \, q_j\neq 2*\odd} \frac{1}{\sqrt{q_j}} \log\frac{q_{j+1}}{q_j} \ll 1+\Sigma_s (x)<+\infty.
\]

\medskip 
 {\bf (3)  $s-1+1/2r_{\odd(x)}=0$ and $s<1$:} Since $h_j\asymp 1/q_jq_{j+1}$, we have 
\[
 \|f\|_{L^2(\widetilde\mu_{H_k})} \ll  H_k^{\frac{s-1/2}2} +\frac{|\theta_0|}{\sqrt{q_j}}   \min\left( H_k^{(s-1)/2},  \frac{H_k}{(q_jq_{j+1})^{(s-3)/2}} \right),
\]
so
\[
 \sum_{q_{j+1}^{-2}\le H_k \le q_j^{-2}}\|f\|_{L^2(\widetilde\mu_{H_k})} \ll (\sum_{q_{j+1}^{-2}\le H_k \le q_j^{-2}} H_k^{\frac{s-1/2}2}) + \frac{|\theta_0|}{\sqrt{q_j}} (\frac{1}{q_jq_{j+1}})^{(s-1)/2}.
\]
Finally,
\[
 \sum_{k\ge 1} \|f\|_{L^2(\widetilde\mu_{H_k})}\ll H^{\frac{s-1/2}2} +\sum_{j,q_j\neq 2*\odd} \sqrt{\frac{q_{j+1}}{(q_jq_{j+1})^s}}\ll 1+\Sigma_s(x)<\infty.
\]

\subsection{Upper bound for the local $L^2$-exponent}

Assume first that $s<1$.

\medskip

 {Let $K$ be a large constant. Let $0<H\le (1/K) q^{-2}$,} with $q\neq 2*\odd$ and $N>H^{-2}$. We apply Propositions \ref{summation_substract} and \ref{L2_average} to get 
\[
\left \|F_{s,N}\left(\frac pq+2\cdot\right)-F_{s,N}\left(\frac pq+\cdot\right) \right\|_{L^2(\widetilde\mu_H)}=\frac{|\theta_0|}{\sqrt q}\left\|{\widetilde F_{s}}(\cdot)\right\|_{L^2(\widetilde\mu_H)}+\mathcal{O}({H^{\frac{s-1/2}2}})
\]
with
\[
 \widetilde{F_{s}}(h)= \int_0^N \frac{e^{4i\pi ht^2}-e^{2i\pi ht^2}}{t^s} \, dt.
\]
Using a change of variable, and then after integrating by parts, one obtains 
\begin{eqnarray*}
 \widetilde{F_{s}}(h) & = & h^{\frac{s-1}2} \int_0^{N\sqrt h} \frac{e^{4i\pi t^2}-e^{2i\pi t^2}}{t^s} \, dt \\
 & = &  h^{\frac{s-1}{2}} \left((2^s-1)\int_0^{+\infty} \frac{e^{2i\pi t^2}}{t^s}dt  + \mathcal{O}\left((N\sqrt{|h|})^{-s-1} \right)\right ).
\end{eqnarray*}

It is easily checked that $\int_0^{+\infty} \frac{e^{2i\pi t^2}}{t^s}dt $ is not zero. This leads us to  the estimate 
\[
 \left\|   \widetilde {F_{s}}(\cdot)   \right\|_{L^2(\widetilde\mu_H)} = C_s   H^{\frac{s-1}2}  \, \left  (1+ \mathcal{O}(H) \right) 
 \]
for some non-zero constant $C_s$. Since  {$0<H\le q^{-2}/ K$}, we deduce that
\begin{equation}\label{lower_bound_rational}
\left \|F_{s,N}\left(\frac pq+2\cdot\right)-F_{s,N}\left(\frac pq+\cdot\right) \right\|_{L^2(\widetilde\mu_H)}\geq  \frac{H^{\frac{s-1}2}}{\sqrt q} 
\end{equation}
when $H$ becomes small enough.

\medskip

Now, pick a convergent $p_j/q_j$   of $x$ with $q_j\neq 2*\odd$, and take $H_j= (1/K) |h_j|$.  One can check that 
$$ H_j  \leq (1/K) \frac{1} {q_j q_{j+1}} \le (1/K) \frac{1} {q_j ^2} .$$
 Then, we   apply \eqref{lower_bound_rational} to obtain that  for every $N\geq H_j^{-2}$, one has
\begin{eqnarray*}
\left \|F_{s,N}\left(\frac {p_j}{q_j}+2\cdot\right)-F_{s,N}\left(\frac {p_j}{q_j}+ \cdot \right) \right\|_{L^2(\widetilde\mu_{H_j})}  & \geq  & \frac{H_j^{\frac{s-1}2}}{\sqrt{q_j}}  = H_j^{\frac{s-1}2} h_j^{1/(2r_j)}\gg H_j^{\frac{s-1+1/r_j}2}.
\end{eqnarray*}

On the other hand, by the triangular inequality, 
\begin{eqnarray*}
\left |F_{s,N}\left(\frac {p_j}{q_j}+2h\right)-F_{s,N}\left(\frac {p_j}{q_j}+h\right)\right |  & \leq &  \left| F_{s,N}\left (\frac {p_j}{q_j}+2h\right)-F_{s,N}(x)\right|\\
 & +&\left|F_{s,N}\left(\frac {p_j}{q_j}+h\right)-F_{s,N}(x)\right|,
\end{eqnarray*}
which implies that   for $\widetilde H_j = H_j$ or $\widetilde H_j = 2H_j$, one has
 \begin{eqnarray*}
\left\|F_{s,N}(x+\cdot)-F_{s,N}(x)   \right\|_{L^2(\widetilde\mu_{\widetilde H_j})}  & \geq &  \frac{1}{2}  \left \|F_{s,N}  \left(\frac {p_j}{q_j}+2\cdot\right)-F_{s,N}\left(\frac {p_j}{q_j}+\cdot\right)   \right\|_{L^2(\widetilde\mu_{H_j})}\\
&\gg &   H_j^{\frac{s-1+1/r_j}2} 
\end{eqnarray*}

Now, we can choose $N$ so large that 
$$\left\|F_{s,N}(x+\cdot)-F_{s,N}(x)   \right\|_{L^2(\widetilde\mu_{\widetilde H_j})} = \left\|F_{s}(x+\cdot)-F_{s}(x)   \right\|_{L^2(\widetilde\mu_{\widetilde H_j})} +{\mathcal{O}(\widetilde H_j),}$$
and we finally obtain
$$\left\|F_{s}(x+\cdot)-F_{s}(x)   \right\|_{L^2(\widetilde\mu_{\widetilde H_j})}  \gg   \widetilde H_j^{\frac{s-1+1/r_j}2}.$$

Since this occurs for an infinite number of $j$, i.e. for an infinite number of small real numbers $\widetilde H_j$ converging to zero, one concludes that 
$$\alpha_{F_s}(x)  \leq \liminf_{j\to+\infty} \frac{s-1+1/r_j}2 = \frac{s-1+1/r_{odd}(x)}2.$$

\end{document}